\documentclass[11pt,reqno]{amsart}

\usepackage{amsfonts,amsmath,amscd,amssymb,amsbsy,amsthm,amstext,amsopn,amsxtra}
 \usepackage{color,fullpage,mathrsfs}
 \usepackage{subfigure}
 \usepackage{verbatim} 
 \usepackage{stmaryrd}
 \usepackage{extarrows}
 \usepackage[all]{xy}
 \usepackage{bm}   
 \usepackage{hyperref}
 \usepackage{enumitem}
\usepackage{cancel}
\usepackage{color}
\usepackage{tikz}
 \usepackage{tikz-cd}

 \numberwithin{equation}{section}
 \hypersetup{colorlinks,linkcolor=[rgb]{0,0,1},citecolor=[rgb]{1,0,0}}

\newtheorem{theorem}{Theorem}
\newtheorem{lemma}[theorem]{Lemma}
\newtheorem{proposition}[theorem]{Proposition}

\theoremstyle{definition}

\newcommand{\CC}{\ensuremath{\mathbb{C}}} 
 
\newcommand{\NN}{\ensuremath{\mathbb{N}}}

\newcommand{\End}{\operatorname{End}}
\newcommand{\git}{\ensuremath{\operatorname{\!/\!\!/\!}}}
\newcommand{\GL}{\operatorname{GL}}
\newcommand{\head}{\operatorname{h}}
\newcommand{\Hilb}{\operatorname{Hilb}}
\newcommand{\Hom}{\operatorname{Hom}}

\newcommand{\Quot}{\operatorname{Quot}}
\newcommand{\Rep}{\operatorname{Rep}}
\newcommand{\SL}{\operatorname{SL}}

\newcommand{\Sym}{\operatorname{Sym}}

\newcommand{\tail}{\operatorname{t}}
\newcommand{\tr}{\operatorname{tr}}

\newcommand{\xpij}[2]{\operatorname{x}_{#1,#2}}

\newcommand{\QuotI}[1][ ]{\Quot_I^{#1}}
\newcommand{\QuotInI}{\QuotI[n_I]}

\title{Orbifold Quot schemes via  the Le Bruyn--Procesi theorem}

\author{Alastair Craw} 

\address{Department of Mathematical Sciences, 
University of Bath, 
Claverton Down, 
Bath BA2 7AY, 
UK.}
\email{a.craw@bath.ac.uk}

\begin{document}
 
 \begin{abstract}
 This note provides a short proof of the fact that the reduced scheme underlying each orbifold Quot scheme associated to a finite subgroup of $\SL(2,\mathbb{C})$ is isomorphic to a Nakajima quiver variety. Our approach uses recent work of the author with Yamagishi, allowing us to bypass the combinatorial arguments and the use of recollement from \cite{CGGS2}. 
 \end{abstract}

 \maketitle

  \section{Introduction}
 Let $\Gamma\subset \SL(2,\CC)$ be a nontrivial, finite subgroup. Orbifold Quot schemes for the Kleinian orbifold $[\mathbb{C}^2/\Gamma]$ were introduced in \cite{CGGS2}, generalising Hilbert schemes of points on the ADE singularity $\CC^2/\Gamma$. Here, we exploit a new approach to variation of GIT quotient for quiver moduli spaces from recent joint work~\cite{CY23} to present a simple proof of \cite[Theorem~1.1(3)]{CGGS2} which establishes that the reduced scheme underlying each orbifold Quot scheme is isomorphic to a Nakajima quiver variety $\mathfrak{M}_{\theta}(1,v)$ for the framed preprojective algebra of $\Gamma$. 
 
  To recall the orbifold Quot schemes, list the irreducible representations of $\Gamma$ as $\rho_0, \rho_1,\dots, \rho_r$, where $\rho_0$ is the trivial representation.
  The group $\Gamma$ acts dually on the coordinate ring $R$ of $\mathbb{C}^2$, giving the decomposition $R\cong\bigoplus_{0\leq i\leq r} R_i\otimes_{\CC} \rho_i$ as a sum of indecomposable $R^\Gamma$-modules, where $R_i=\Hom_\Gamma(\rho_i,R)$ for $0\leq i\leq r$. For a non-empty subset $I\subseteq \{0,1,\dots, r\}$, consider the submodule 
  \[
  R_I:=\bigoplus_{i\in I} R_i
  \]
  of $R$. 
  The noncommutative algebra $\End(R_I)$ can be constructed as the quotient of the path algebra of a quiver with vertex set $I$ by a two-sided ideal of relations \cite[(3.5)]{CGGS1}. For each $i\in I$, if we write $e_i:= \text{id}_{R_i}\in \End(R_I)$ for the vertex idempotent determined by the identity map, then the dimension vector of a finite-dimensional $\End(R_I)$-module $Z$ is $\dim Z:=(\dim e_i Z)_{i\in I}\in \NN^I$. As a set, the \emph{orbifold Quot scheme} for $\Gamma$ associated to $I$ and any vector $n_I=(n_i)_{i\in I}\in \mathbb{N}^{I}$ is     
   \[
   \Quot_I^{n_I}\big([\mathbb{C}^2/\Gamma]\big):= \big\{\End(R_I) \text{-epimorphisms } R_I\twoheadrightarrow Z 
   \mathrel{\big|} 
   \dim Z = n_I\big\}.
   \]
   A key result from \cite[Proposition~3.2]{CGGS2} shows that $\Quot_I^{n_I}\big([\mathbb{C}^2/\Gamma]\big)$ can be constructed as a scheme of finite type over $\CC$ that represents a functor, denoted $\mathcal{Q}^{n_I}_{\End(R_I)}(R_I)$. 
   
   Here we prove the following result that streamlines slightly the statement of \cite[Theorem~1.1(3)]{CGGS2}:
  
  \begin{theorem}
\label{thm:main}
 Let $I\subseteq \{0,1,\dots, r\}$ be a non-empty subset and let $n_I=(n_i)_{i\in I}\in \mathbb{N}^{I}$. There exists $v:=v(n_I) \in \mathbb{N}^{r+1}$ and a stability condition $\theta_I$ for the framed preprojective algebra of $\Gamma$ such that 
 \[
\QuotInI ([\mathbb{C}^2/\Gamma])_{\mathrm{red}} \cong  \mathfrak{M}_{\theta_I}(1,v).
\]
 In particular, $\QuotInI ([\mathbb{C}^2/\Gamma])$ is irreducible, and its underlying reduced scheme is normal and has symplectic singularities. Moreover, it admits at least one projective symplectic resolution.
\end{theorem}

Our approach bypasses the combinatorial arguments and the use of recollement in the original proof, so we also bypass the gap in the proof of \cite[Proposition~6.1]{CGGS2} that was pointed out to us by Yehao Zhou (see \cite[Remarks~6.8(2)]{CY23}). An alternative approach to this can be found in \cite{Bertsch24}. 

In the special case when there exists $n\in \mathbb{N}$ satisfying $n_i=n\dim \rho_i$ for all $i\in I$, it was shown in \cite[Corollary~6.7]{CY23} that the scheme $\QuotInI ([\mathbb{C}^2/\Gamma])$ is actually reduced, so this orbifold Quot scheme with its natural scheme structure is isomorphic to a Nakajima quiver variety. Here, in the proof of the more general statement from Theorem~\ref{thm:main}, we apply a result of Crawley-Boevey that forces us to work with the underlying reduced scheme structure. 

\smallskip

\noindent \textbf{Notation:} For any scheme $X$, we write $X_{\textrm{red}}$ for the underlying reduced subscheme.

 \section{Proof of Theorem~\ref{thm:main}}
 Let $n\in \mathbb{N}$ be the largest non-negative integer with $n\dim(\rho_i) \leq n_i$ for all $i\in I$. Define the vector $v=(v_i)\in \mathbb{N}^{r+1}$ by setting 
 \[
 v_i=\left\{
 \begin{array}{cr} 
 n_i & \text{for }i\in I, \\
 n\dim(\rho_i) & \text{otherwise.}
 \end{array}\right.
 \]
 The McKay quiver $Q_\Gamma$ is the doubled quiver of an extended Dynkin graph of type ADE with vertex set $\{0,1,\dots, r\}$, where $0$ is the extending node. The preprojective algebra $\Pi_\Gamma$ is the quotient of the path algebra of $Q_\Gamma$ by the two-sided ideal of preprojective relations. 
 
 The symplectic vector space $\Rep(Q_\Gamma,v)$ of representations of the quiver $Q_\Gamma$ with dimension vector $v$ admits a Hamiltonian action by  \[
 G(v):=\prod_{0\leq i\leq r} \GL(v_i).
 \]
 Write $\Rep(\Pi_\Gamma,v)$ for the $G(v)$-invariant, affine subscheme of $\Rep(Q_\Gamma,v)$ parametrising  representations of $Q_\Gamma$ that satisfy the preprojective relations; this subscheme is often denoted $\mu_v^{-1}(0)$, where $\mu_v$ is the moment map. 
 
 When $n_i=n\dim(\rho_i)$ for all $0\leq i\leq r$,  then $v=n\delta$ where $\delta\in\mathbb{N}^{r+1}$ is the minimal imaginary root in the root system of type ADE associated to $\Gamma$. In this case, we write $\mu_{n\delta}^{-1}(0)$ and $G(n\delta)$ for the corresponding representation scheme and reductive group respectively.
 
 \begin{lemma}
 \label{lem:RepalphatoSym}
 For any $0\leq i\leq r$, the $\CC$-algebra $\CC[\Rep(\Pi_\Gamma,v)_{\mathrm{red}}]^{G(v)}\cong \CC[\Sym^n(\mathbb{C}^2/\Gamma)]$ is generated by the trace functions $\tr_p$ associated to cycles in $Q_\Gamma$ with head and tail at vertex $i$.
 \end{lemma}
 \begin{proof}
 Apply \cite[Proposition~4.1]{CY23} to see that $\CC[\mu_{n\delta}^{-1}(0)]^{G(n\delta)} \cong \CC[\Sym^n(\mathbb{C}^2/\Gamma)]$ is generated by the trace functions $\tr_p$ associated to cycles in $Q_\Gamma$ with head and tail at $i$. More generally, by taking the reduced scheme structure on $\Rep(\Pi_\Gamma,v)$, we may apply  \cite[Lemma~2.3]{CB02} repeatedly to see that the $\CC$-algebra homomorphism 
 \[
 \CC[\Rep(\Pi_\Gamma,v)_{\mathrm{red}}]^{G(v)}\rightarrow \CC[\mu_{n\delta}^{-1}(0)]^{G(n\delta)}
 \]
 given by sending each trace function $\tr_p$ for dimension $v$ to the corresponding trace function $\tr_p$ for dimension $n\delta$ is an isomorphism.
 \end{proof}
 
 Let $\Pi$ denote the preprojective algebra of the framed McKay quiver $Q$ obtained from the McKay quiver $Q_\Gamma$ by adding a framing vertex $\infty$ and an arrow in each direction between $\infty$ and $0$ (see \cite{CGGS1}).  Regard $\alpha:=(1,v)\in \mathbb{N}\oplus \mathbb{N}^{r+1}$ as a dimension vector for both $\Pi$ and $A:= \Pi/(b^*)$, where $b^*$ is the unique arrow in $Q$ with head at $\infty$. 
 
 Following \cite[Proposition~3.3]{CGGS1}, there is a quiver $Q^*:=Q_{\{0,1,\dots,r\}}^*$ and an epimorphism $\beta\colon \mathbb{C}Q^*\to A$ of $\CC$-algebras. The space $\Rep(Q^*,\alpha)$ of representations of $Q^*$ with dimension vector $\alpha$ contains the closed subscheme $\Rep(A,\alpha)$ of representations satisfying the relations $\ker(\beta)$. The group 
 \[
 G(\alpha):=\GL(1)\times \prod_{0\leq i\leq r} \GL(v_i)
 \]
  acts on $\Rep(A,\alpha)$ by change of basis for representations. Define $K:= \{0,1,\dots, r\}\setminus I$, and regard $I$ and $K$ as subsets of the vertex set $Q_0^*=\{\infty,0,1,\dots, r\}$, so 
  \[
  Q_0^*=\{\infty\}\sqcup I\sqcup K.
  \]
  Note that the subgroup $H_K:= \prod_{k\in K} \GL(v_k)$ of $G(\alpha)$ acts on $\Rep(A,\alpha)$.
 
 Next, regard $\alpha_I:=(1,n_I)\in \mathbb{N}\times \mathbb{N}^I$ as a dimension vector for the subalgebra $A_I\subseteq A$ spanned by classes of paths with head and tail in $I\subseteq Q_0$.  Again, \cite[Proposition~3.3]{CGGS1} gives a quiver $Q_I^*$ with vertex set $\{\infty\}\cup I$ and a $\CC$-algebra epimorphism $\beta_I\colon \mathbb{C}Q_I^*\to A_I$, and we write $\Rep(A_I,\alpha_I)$ for the affine scheme of representations of $Q_I^*$ that have dimension $\alpha_I$ which satisfy the relations $\ker(\beta_I)$. The group 
 \[
 G(\alpha_I):= \GL(1)\times \prod_{i\in I} \GL(n_i)
 \]
 acts on $\Rep(A_I,\alpha_I)$. Taking the underlying reduced scheme structure defines a (possibly reducible) variety with coordinate ring $\CC[\Rep(A_I,\alpha_I)_{
\mathrm{red}}]$.
   
 \begin{proposition}
 \label{prop:psired}
 There is a $\CC$-algebra isomorphism $\CC[\Rep(A_I,\alpha_I)_{
\mathrm{red}}]\cong \CC[\Rep(A,\alpha)_{\mathrm{red}}]^{H_K}$.
 \end{proposition}
 \begin{proof}
 The proof builds on that of \cite[Theorem~5.3]{CY23}. For a nontrivial path $p=a_\ell\cdots a_1$ in $Q^*$ with head and tail at $\head(p)$ and $\tail(p)$ respectively, and for indices $1\leq i\leq \alpha_{\head(p)}$ and $1\leq j\leq \alpha_{\tail(p)}$, let $\xpij{p}{ij}\in\CC[\Rep(Q^*,\alpha)]$ denote the function that sends $(B_a)\in \Rep(Q^*,\alpha)$ to the $(i,j)$-entry of the matrix $B_{a_\ell}\cdots B_{a_1}$. The construction from  
 \cite[Proposition~3.3]{CGGS1} shows that $Q_I^*$ is a subquiver of $Q^*$.  The image of the resulting inclusion of polynomial rings 
 \[
 \overline{\psi}\colon\CC[\Rep(Q_I^*,\alpha_I)]\hookrightarrow \CC[\Rep(Q^*,\alpha)]
 \]
 is the polynomial subalgebra $R\subset \CC[\Rep(Q^*,\alpha)]^{H_K}$ generated by the functions $\xpij{a}{ij}$, where $a\in Q^*_1$ has head and tail in $\{\infty\}\cup I$, for $1\leq i\leq \alpha_{\head(a)}$ and $1\leq j\leq \alpha_{\tail(a)}$.  
 
  As in the proof of \cite[Theorem~5.3]{CY23}, the map $\overline{\psi}$ induces a $\CC$-algebra homomorphism
 \[
\psi\colon \CC[\Rep(A_I,\alpha_I)]\longrightarrow \CC[\Rep(A,\alpha)]^{H_K}
\]
whose image is generated by the classes in $\CC[\Rep(A,\alpha)]^{H_K}$ of the functions $\xpij{a}{ij}$ for arrows $a$ in $Q^*$ that have head and tail in $\{\infty\}\cup I$, for $1\leq i\leq \alpha_{\head(a)}$ and $1\leq j\leq \alpha_{\tail(a)}$. The construction of the right-inverse to $\psi$ from the proof of \cite[Theorem~5.3]{CY23} shows that $\psi$ is injective. Taking the quotient by the nilradical $\sqrt{0}$ gives 
\[
\CC[\Rep(A,\alpha)]^{H_K}/\sqrt{0}\cong \CC[(\Rep(A,\alpha)\git H_K)_{\mathrm{red}}]\cong 
 \CC[\Rep(A,\alpha)_{\mathrm{red}}]^{H_K},\]
 and we see that $\psi$ induces a $\CC$-algebra homomorphism 
 \[
  \psi_{\mathrm{red}}\colon \CC[\Rep(A_I,\alpha_I)]_{\mathrm{red}}\longrightarrow \CC[\Rep(A,\alpha)_{\mathrm{red}}]^{H_K}.
  \]
 Injectivity of $\psi$ implies that $\psi_{\mathrm{red}}$ is injective. It remains to show that $\psi_{\mathrm{red}}$ is surjective. 

Apply \cite[Proposition~3.2]{CY23} to see that $\CC[\Rep(A,\alpha)]^{H_K}$ is generated by the trace functions $\tr_p$ of cycles $p$ in $Q^*$ traversing only arrows with head and tail in $K$, together with the contraction functions $\xpij{p}{ij}$ for paths $p$ with head and tail in $\{\infty\}\cup I$ for indices  $1\leq i\leq \alpha_{\head(p)}$ and $1\leq j\leq \alpha_{\tail(p)}$. We claim that the trace function generators are redundant. Indeed, as in the proof of \cite[Proposition~4.5]{CY23}, the class of each nontrivial cycle in $Q^*$ may be regarded as the class in $A$ of a cycle that does not touch vertex $\infty$, so each trace function may be regarded as an element of $\CC[\Rep(\Pi_\Gamma,v)]^{G(v)}$. Again, taking the quotient by the nilradical allows us to 
work in the algebra 
\[
\CC[\Rep(\Pi_\Gamma,v)]^{G(v)}/\sqrt{0}\cong \CC[(\Rep(\Pi_\Gamma,v)\git G(v))_{\mathrm{red}}]\cong \CC[\Rep(\Pi_\Gamma,v)_{\mathrm{red}}]^{G(v)}.
\]
Lemma~\ref{lem:RepalphatoSym} shows that this algebra is generated by trace functions associated to cycles $q$ with head and tail at a vertex $i\in I$. Every such trace function equals $\sum_{1\leq j\leq \alpha_i} \xpij{q}{jj}$, i.e., the trace function generators can be written in terms of the contraction function generators. This proves the claim. 

The construction of the quiver $Q^*= Q_{\{0,1,\dots,r\}}^*$ from \cite[Proposition~3.2]{CGGS1} shows that every path in $Q^*$ with head and tail in $\{\infty\}\cup I$ is equivalent modulo $\ker(\beta)$ to a linear combination of paths that traverse only arrows with head and tail in $\{\infty\}\cup I$. Therefore $\CC[\Rep(A,\alpha)]^{H_K}$ is generated by the functions $\xpij{a}{ij}$ for arrows $a$ in $Q^*$ with head and tail in $\{\infty\}\cup I$ for indices  $1\leq i\leq \alpha_{\head(p)}$ and $1\leq j\leq \alpha_{\tail(p)}$. Therefore $\psi_{\mathrm{red}}$ is surjective and hence an isomorphism as required.
\end{proof}
  
  \begin{proof}[Proof of Theorem~\ref{thm:main}]
  For the character $\eta_I\in G(\alpha_I)^\vee$ from \cite[(4.2)]{CGGS2}, the orbifold Quot scheme from Theorem~\ref{thm:main} satisfies $\QuotInI ([\mathbb{C}^2/\Gamma]) \cong \Rep(A_I,\alpha_I)\git_{\eta_I} G(\alpha_I)$ by \cite[Proposition~4.2]{CGGS2}, giving 
 \[
 \QuotInI ([\mathbb{C}^2/\Gamma])_{\mathrm{red}} \cong \Rep(A_I,\alpha_I)_{\mathrm{red}}\git_{\eta_I} G(\alpha_I).
 \]
The quotient map $G(\alpha)\to G(\alpha)/H_K\cong G(\alpha_I)$ induces the inclusion of character groups that sends $\eta_I\in G(\alpha_I)^\vee$ to the character $\theta_I\in G(\alpha)^\vee$ from \cite[(5.1)]{CGGS2}. The construction from \cite[Section~5.2]{CY23} now shows that the isomorphism $\psi_{\mathrm{red}}$ of $\CC$-algebras induces an isomorphism
\[
\Rep(A_I,\alpha_I)_{\mathrm{red}}\git_{\eta_I} G(\alpha_I)\cong \Rep(A,\alpha)_{\mathrm{red}}\git_{\theta_I} G(\alpha)
 \]
of schemes over $\Rep(A,\alpha)_{\mathrm{red}}\git_0 G(v)\cong\Sym^n(\mathbb{C}^2/\Gamma)$. For the vector $\alpha=(1,v)$, the isomorphism \[
\Hilb^v([\mathbb{C}^2/\Gamma])\cong \mathfrak{M}_{\theta}(\alpha)
\]
from \cite[Proposition~5.2]{CGGS2} allows us to apply the proof of \cite[Lemma~3.1]{CGGS1}, using $v$ in place of $n\delta$ throughout, to conclude that the (reduced by assumption) Nakajima quiver variety $\mathfrak{M}_{\theta_I}(\alpha)$ may be regarded as a (reduced) moduli space of $A$-modules, giving 
\[
\Rep(A,\alpha)_{\mathrm{red}}\git_{\theta_I} G(\alpha)\cong \mathfrak{M}_{\theta_I}(\alpha).
\]
Combining these isomorphisms gives $ \QuotInI ([\mathbb{C}^2/\Gamma])_{\mathrm{red}} \cong\mathfrak{M}_{\theta_I}(\alpha)$ as required.
\end{proof}

 \small{

}

\end{document}